\documentclass{elsarticle}
\usepackage{amsfonts, amsmath, amssymb, amsthm}
\usepackage[T1]{fontenc} %
\usepackage[latin1]{inputenc} %
\usepackage{graphicx}
\usepackage{subfigure}
\usepackage{color, fancybox, calc}
\usepackage{latexsym,amssymb,amsmath}
\usepackage{fullpage}
\usepackage{array,url, appendix}
\usepackage{refcount}

\newtheorem{theorem}{Theorem}
\newtheorem{lemma}[theorem]{Lemma}

\renewcommand{\deg}{\mathrm{deg}}

\begin{document}
\title{The Erd\H{o}s-Hajnal Conjecture for Paths and Antipaths}

\author[lirmm]{N. Bousquet}
\ead{nicolas.bousquet@lirmm.fr}

\author[ens]{A. Lagoutte\fnref{ANR} \corref{cor1}}
\ead{aurelie.lagoutte@ens-lyon.fr}

\author[ens]{S. Thomass\'e\fnref{ANR}}
\ead{stephan.thomasse@ens-lyon.fr}

\cortext[cor1]{Corresponding author}

\fntext[ANR]{These authors were partially supported by ANR Projet \textsc{Stint} under Contract \textsc{ANR-13-BS02-0007}.}

\address[lirmm]{AlGCo project-team, CNRS, LIRMM, 
161 rue Ada, 34392 Montpellier France.}
\address[ens]{LIP, UMR 5668 ENS Lyon - CNRS - UCBL - INRIA, Universit\'e de Lyon, 46, all\'ee d'Italie, 69364 Lyon France.}

\date{\today}                                 %

\begin{frontmatter}

\begin{abstract} 
We prove that for every $k$, there exists $c_k>0$ such that 
every graph $G$ on $n$ vertices with no induced path $P_k$ or its complement $\overline {P_k}$
contains a clique or a stable set of size $n^{c_k}$.
\end{abstract}


\end{frontmatter}

An \emph{$n$-graph} is a graph on $n$ vertices. For every vertex $x$, $N(x)$ denotes the neighborhood of $x$, that is the set of vertices $y$ such that $xy$ is an edge. The degree $\deg(x)$ is the size of $N(x)$. In this note, we only consider classes of graphs that are closed under
induced subgraphs. Moreover a class $\mathcal C$  is \emph{strict} if it does not contain all graphs. It is said  to have the \emph{(weak) Erd\H{o}s-Hajnal property} if there 
exists some $c>0$ such that every graph of $\mathcal C$ contains a clique or a stable set 
of size $n^c$ where $n$ is the size of $G$. The Erd\H{o}s-Hajnal conjecture~\cite{Erdo89} asserts that every strict class of graphs
has the Erd\H{o}s-Hajnal property; see~\cite{Chud13} for a survey. This fascinating 
question is open even for graphs not
inducing a cycle of length five. When excluding a single graph $H$, Alon, Pach and Solymosi
showed in~\cite{Alon01} that it suffices to consider \emph{prime} $H$, namely graphs without nontrivial modules (a \emph{module} is a subset $V'$ of vertices such that for every $x,y\in V'$, $N(x)\setminus V'=N(y)\setminus V'$).
A natural approach is then to study classes of graphs with intermediate difficulty, hoping to get a proof scheme which could be extended. A natural prime candidate to forbid is certainly 
the path. Unfortunately, even excluding the path on five vertices seems already hard. Chudnovsky and Zwols studied the class ${\mathcal C}_k$ of graphs not inducing the path $P_k$ on $k$ 
vertices or its complement $\overline {P_k}$. They proved the Erd\H{o}s-Hajnal property for $P_5$ and $\overline {P_6}$-free graphs~\cite{Chud12}. 
This was extended for $P_5$ and $\overline {P_7}$-free graphs by Chudnovsky and Seymour~\cite{ChudSey13}. Moreover structural results have been provided for $\mathcal{C}_5$ \cite{ChudMac13, ChudMacPen13}. We show in this note 
that for every fixed $k$, the class ${\mathcal C}_k$ has the Erd\H{o}s-Hajnal property.
An $n$-graph is an {\it $\varepsilon$-stable set} if it has at most $\varepsilon {n\choose 2}$ edges. The complement of an $\varepsilon$-stable set is 
an {\it $\varepsilon$-clique}. 
Fox and Sudakov \cite{Fox08} proved the following:

\begin{theorem}[\cite{Fox08}] \label{reg}
For every positive integer $k$ and every $\varepsilon \in (0,1/2)$, there exists $\delta>0$ such that every $n$-graph 
$G$ satisfies one of the following:
\begin{itemize}
\item $G$ induces all graphs on $k$ vertices.
\item $G$ contains an $\varepsilon$-stable set of size at least $\delta n$.
\item $G$ contains an $\varepsilon$-clique of size at least $\delta n$.
\end{itemize}
\end{theorem}


Note that a stronger result was previously showed by R\"odl~\cite{Rodl86} using Szemer\'edi's regularity lemma, but Fox and Sudakov's proof provides a much better quantitative estimate  ($\delta=2^{-ck (\log 1/\varepsilon)^2}$ for some constant $c$). They further conjecture that a polynomial estimate should hold, which would imply the Erd\H{o}s-Hajnal conjecture.

In a graph $G$, a \emph{biclique of size $t$} is a (not necessarily induced) complete bipartite subgraph $(X,Y)$ such that both $ |X|, |Y|\geq t$. Observe that it does not require any condition
inside $X$ or inside $Y$.
 Erd\H{o}s, Hajnal and Pach proved in~\cite{Erdo00} that 
for every strict class $\mathcal C$, there exists some $c>0$ such that for every  $n$-graph $G$ in
$\mathcal C$, $G$ or its complement $\overline{G}$ contains a biclique of size $n^c$. This "half" version
of the conjecture was improved to a "three quarter" version by Fox and Sudakov~\cite{Fox09}, 
where they show the existence of a polynomial size stable set or biclique.
Following the notations of \cite{FoxPach08}, a class $\mathcal{C}$ of graphs has the \emph{strong Erd\H{o}s-Hajnal property} if there exists a constant $c$ such that for every $n$-graph $G$ in $\mathcal{C}$, $G$ or $\overline{G}$ contains a biclique of size $cn$.
It was proved that having the strong Erd\H{o}s-Hajnal property implies having the (weak) Erd\H{o}s-Hajnal property:

\begin{theorem} [\cite{Alon05, FoxPach08}]\label{bipartite}
If $\mathcal C$ is a class of graphs having the strong Erd\H{o}s-Hajnal property,
then $\mathcal C$ has the weak Erd\H{o}s-Hajnal property.
\end{theorem}

\begin{proof} (sketch)
Let $c$ be the constant of the strong Erd\H{o}s-Hajnal property, meaning that for every $n$-graph $G$ in $\mathcal{C}$, $G$ or $\overline{G}$ contains a biclique of size $cn$.
Let $c'>0$ be such that $c^{c'}\geq 1/2$. We prove
by induction that every $n$-graph $G$ in $\mathcal C$ induces a $P_4$-free
graph of size $n^{c'}$. By our hypothesis on $\mathcal C$, there
exists, say, a biclique $(X,Y)$ of size $cn$ in $G$. Applying the induction hypothesis inside both $X$ and $Y$,
we form a $P_4$-free graph on $2(cn)^{c'}\geq n^{c'}$ vertices. The 
Erd\H{o}s-Hajnal property of $\mathcal C$ follows from the fact that 
every $P_4$-free $n^{c'}$-graph has a clique or a stable set 
of size at least $n^{c'/2}$.
\end{proof}

We now prove our main result. The key lemma is an adaptation of Gy\'arf\'as'
proof of the $\chi$-boundedness of $P_k$-free graphs, see~\cite{Gyar87}.


\begin{lemma} \label{lemme chemin}
For every $k\geq 2$, there exists $\varepsilon_k>0$ and $c_k$ (with $0<c_k\leq 1/2$) such that every connected $n$-graph $G$ with $n\geq 2$ satisfies one of the following:

\begin{itemize}
\item There exists a vertex of degree more than $ \varepsilon_k n$.
\item For every vertex $v$, $G$ contains an induced $P_k$ starting at $v$.
\item The complement $\overline{G}$ of $G$  contains a biclique of size $c_k n$.
\end{itemize}

\end{lemma}

\begin{proof}
We proceed by induction on $k$. For $k=2$, since $G$ is connected, every vertex is the endpoint of an edge (that is, a $P_2$). Thus we can arbitrarily define $\varepsilon_2=c_2=1/2$.

If $k>2$, let $\varepsilon_k= \frac{\varepsilon_{k-1}}{(2+\varepsilon_{k-1})}$ and $c_k=\frac{c_{k-1}(1-\varepsilon_k)}{2}$. Let us assume that the first item is false. We will show that the second or the third item is true. Let $v_1$ be any vertex and $S=V(G)\setminus (N(v_1)\cup \{v_1\})$. The size $s$ of $S$ is at least $ (1-\varepsilon_k)n-1$. If $S$ have only \emph{small} connected components, meaning of size at most $s/2$, then one can divide the connected components into two parts with at least $(s+1)/4$ vertices each, and no edges between both parts. This gives in $\overline{G}$ a biclique of size $(s+1)/4\geq \frac{(1-\varepsilon_k)n}{4}$, thus of size at least $c_k n$ since $c_k\leq \frac{1-\varepsilon_k}{4}$.
Otherwise, $S$ has a \emph{giant} connected component $S'$, meaning of size $s'$ more than $s/2$. Let $v_2$ be a vertex adjacent both to $v_1$ and to some vertex in $S'$. Observe that $v_2$ exists since $G$ is connected.
Consider now the graph $G_2$ induced by $S'\cup \{v_2\}$.
The maximum degree in $G_2$ is still at most $\varepsilon_k n= \varepsilon_{k-1} (1-\varepsilon_k) n/2 \leq \varepsilon_{k-1} (s'+1)$.
By the induction hypothesis, either the second or the third item is true for $G_2$ with parameter $k-1$. The second item gives an induced $P_{k-1}$ in $G_2$  starting at $v_2$, thus an induced $P_k$ in $G$  starting at $v_1$. The third item gives a biclique of size $c_{k-1} |G_2|$ in $\overline{G_2}$.
 Since $|G_2|= s'+1 \geq \frac{1-\varepsilon_k}{2} n$, this gives a biclique of size at least $\frac{c_{k-1}(1-\varepsilon_k)}{2} n=c_k n$
 and concludes the proof.
%
\end{proof}

\begin{theorem} \label{path}
For every $k\geq 2$, $\mathcal{C}_k$ has the strong Erd\H{o}s-Hajnal property.
Thus, by Theorem~\ref{bipartite},
the class  ${\mathcal C}_k$ has the (weak) Erd\H{o}s-Hajnal property.


\end{theorem}

\begin{proof} Let $\varepsilon_k$ be as defined in Lemma \ref{lemme chemin}  and $\varepsilon=\varepsilon_k/8>0$. By Theorem~\ref{reg}, there exists
$\delta >0$ such that every graph $G$ not inducing 
$P_k$ or $\overline{P_k}$ does contain an $\varepsilon$-stable set
or an $\varepsilon$-clique of size at least $\delta n$. Free to consider the complement of $G$, we can assume that 
$G$ contains an $\varepsilon$-stable set $S_0$ of size $\delta n$.
We start by deleting in $S_0$ all the vertices
with degree in $S_0$ at least $2\varepsilon s_0$ where $s_0$ is the size of $S_0$. Since the average 
degree in $S_0$ is at most $\varepsilon s_0$, we do not delete more than half of the vertices.
We call $S$ the remaining subgraph which is a $4\varepsilon$-stable set 
of size $s\geq \delta n/2$ with maximum degree less than $4\varepsilon s$. 

Let $G_S$ be the graph induced by $S$. Our goal is to find a constant $c$ such that $\overline{G_S}$ have a biclique of size $cs$, which gives a biclique in $\overline{G}$ of size at least $ c\delta n/2$ and concludes the proof. 
Assume first that $G_S$ only has \emph{small} connected components, meaning of size less than $s/2$. Then one can partition the connected components of $G_S$ in order to get a biclique in $\overline{G_S}$ of size $s/4$.
Otherwise, $G_S$ has a connected component $S'$ of size $s'\geq  s/2$. The degree of every vertex in $S'$ is at most $8 \varepsilon s'=\varepsilon_k s'$, and $S'$ does not contain any induced $P_k$ since $G$ does not. By Lemma \ref{lemme chemin}, there exists a biclique of size $c_k s'\geq c_k s/2$ in the complement of the graph induced by $S'$, thus in $\overline{G_S}$.
%
\end{proof}

\bibliographystyle{plain}

\end{document}